\documentclass[12pt,reqno]{amsart}
\usepackage[cp1251]{inputenc}
\usepackage[english]{babel}
\usepackage[breaklinks=true,colorlinks=true,linkcolor=blue,citecolor=blue,urlcolor=blue]{hyperref}
\usepackage{amsmath,amsopn,amssymb,amsthm}

\newtheorem{theorem}{Theorem}

\begin{document}

\title[On the geodesic diameter of surfaces \dots]
{On the geodesic diameter of surfaces \\ with involutive isometry}
\author{Yu.G. Nikonorov}

\begin{abstract}
This is an English translation of the following paper, published several years ago:
{\it
Nikonorov Yu.G.
On the geodesic diameter of surfaces with involutive isometry (Russian),
Tr. Rubtsovsk. Ind. Inst., 2001, V. 9, 62--65, Zbl. 1015.53041}.
All inserted footnotes provide additional information related to the mentioned problem.

\vspace{2mm} \noindent 2010 Mathematical Subject Classification:
51M16, 53C45, 57M60.

\vspace{2mm} \noindent Key words and phrases: geodesic diameter, length metric space, intrinsic distance, surface, involutive isometry.
\end{abstract}

\maketitle

Let us consider a length metric space $(M,\rho)$, i.~e. the metric
$\rho$ on $M$ is intrinsic~\cite{Bak}.
This means that for every two points $x,y \in M$, the distance $\rho (x,y)$
is equal to the infimum of the lengths of all paths $\gamma \subset M$, connected the points
$x$ and $y$. For a curve $\gamma$, we denote by $l(\gamma)$ the length of $\gamma$ with respect to the metric $\rho$.
The shortest path is a curve $\gamma \subset M$ homeomorphic to the interval $[0,1]\subset \mathbb{R}$ and such that
$l(\gamma)=\rho(x,y)$,
where the points
$x$ and $y$ are the endpoints of the curve $\gamma$.
If a  length metric space $(M,\rho)$ is compact, then for every points $x,y \in M$ there is
a shortest path $\gamma$, connected them \cite{Bak}.
The value $\operatorname{diam}(M)=\sup\, \rho(x,y)$, where $x$ and $y$ are taken in $M$, is called the geodesic diameter of the space $(M,\rho)$.
If $(M,\rho)$ is compact, then $\operatorname{diam}(M)$ could be calculated as the maximal distance between
two points $x,y \in M$.

Note that we can consider
the boundary of a convex body in Euclidean space with an intrinsic metric,
induced by Euclidean one, as the space $(M, \rho)$. Even due this type of examples, it is clear that
calculating of the geodetic diameter is an extremely difficult problem.
For example, if we consider the boundary (surface) of a rectangular three-dimensional parallelepiped,
then the farthest point (in terms of intrinsic boundary metrics)
from a given vertex of the parallelepiped
is not necessarily the opposite vertex \cite{Nik}.\footnote{See also
{\it  Vyalyi M.N. The shortest paths along the parallelepiped surface (Russian) // Mat. Pros., Ser. 3, 9 (2005), 203--206;
Miller S.M., Schaefer E.F. The distance from a point to its opposite along the surface of a box // Pi Mu Epsilon J. 14(2) (2015),  143-154, arXiv:1502.01036.}}
However, the presence of some symmetry in the space $(M, \rho)$
allows to simplify the calculation of the geodetic diameter.

For a given metric space $(M, \rho)$, a map $I:M\rightarrow M$  is called an isometry if
$\rho (I(x),I(y))=\rho (x,y)$ for all $x,y\in M$.

An isometry $I:M\rightarrow M$ is called involutive if the equality  $I\circ I=\operatorname{Id}$ fulfilled, i.~e.
we get the identical map twice applying $I$.
In the case when the isometry has no fixed point ($I(x)\neq x$ for all
$x\in M$), points of the form $x$ and $I(x)$ will be called antipodal.
Recall that $y=I(x)$ implies $x=I(y)$.

Examples of metric spaces with involutive isometry are
surfaces bounding centrally symmetric convex bodies
in Euclidean space.
Involutive isometry in this case is the restriction of the central
symmetry to the boundary of the body under consideration.
It is natural to assume that the geodesic diameter of such
spaces can be calculated as the maximum distance between pairs
of antipodal points.

We have the following

\begin{theorem}
Let $(M,\rho)$ be a length metric space that is homeomorphic to the two-dimensional sphere $S^2$ and let $I:M\rightarrow M$ be an involutive isometry
without fixed points. Then there is $x\in M$ such that
$\operatorname{diam}(M)=\rho(x,I(x))$.\footnote{\, In the paper {\it V\^{\i}lcu C. On two conjectures of Steinhaus // Geometriae Dedicata, 79 (2000), 267--275},
the following version of this theorem is proved (see Proposition 6):
{\it Let $F$ be a convex centrally symmetric surface in $\mathbb{E}^3$, and $\operatorname{diam}(F)$
its intrinsic diameter. If $x,y \in F$ are such that $\operatorname{diam}(F)$ is equal to the intrinsic distance
between $x$ and $y$ on $F$, then the points $x$ and $y$ are symmetric to
each other
under the central symmetry of the surface $F$.}}
\end{theorem}

\begin{proof}
It is easy to see that the map $I$ is a homeomorphism.
Let us consider points $x$ and $y$ such that $\operatorname{diam}(M)=\rho(x,y)$.
Such points do exist since the space $(M, \rho)$ is compact.
Suppose that $y\neq I(x)$ and consider the shortest curve
${\gamma}_{x,I(x)}$,
connecting $x$ and $I(x)$.
Note that the point $y$ (as well as its image $I(y)$ under the map $I$) could not be situated on the considered shortest curve.
Indeed, otherwise the inequality
$\rho(x,y)<\rho(x,I(x))\leq \operatorname{diam}(M)$ holds that is impossible.
Let us consider on the curve
${\gamma}_{x,I(x)}$
all pairs of points $(u,v)$ such that $v=I(u)$.
The set of such pairs is not empty, since the pair $(x,I(x))$ has a suitable property. Among all pairs with this property,
we chose the pair $(\widetilde{u},\widetilde{v})$ with the minimal value of $\rho(u,v)$.
Such pair exists due to the compactness of ${\gamma}_{x,I(x)}$.

Let $\gamma$ be a part of the curve
${\gamma}_{x,I(x)}$ between the chosen points $\widetilde{u}$ and $\widetilde{v}$.
It is easy to see that every points $u$ and $v$  on the curve
$\gamma$ with the property $v=I(u)$ are the endpoints of $\gamma$.
Indeed, the existence of some other such points contradicts to the choice of the pair
$(\widetilde{u},\widetilde{v})$.
Therefore, if we consider the curve $\widetilde{\gamma}$, the image of the curve $\gamma$ under the map $I$, then
$$
\widetilde{\gamma} \cap \gamma=\{\widetilde{u}, \widetilde{v}\}.
$$
Hence, the curve $\eta=\widetilde{\gamma} \cup \gamma$ is homeomorphic to the circle $S^1$,
and according to the Jordan curve theorem, it separates $M$ into two  regions homeomorphic to the two-dimensional disc (open ball),
i.~e. $M=\eta \cup D_1 \cup D_2$,
where each $D_i$ is homeomorphic to $D^2$.

It is clear that the curve $\eta$ does not contain the points $y$ and $I(y)$
(this contradicts to the choice of the points $x$ and $y$).
Without loss of generality, we may assume that $y \in D_1$.
Let us show that $I(y)\in D_2$. Indeed, since $I(\eta)=\eta$,
then either $I(D_1)=D_1$, or $I(D_1)=D_2$.
In the first case we get $I(\eta \cup D_1)=\eta \cup D_1$.
Since $\eta \cup D_1$ is homeomorphic to the closed two-dimensional ball,
then by Brouwer's fixed-point theorem, there is a point
$z\in \eta \cup D_1$ such that $I(z)=z$. The latter equality contradicts to the assumption that
the map $I$ has no fixed point. Therefore, we have
$I(D_1)=D_2$ and $I(y)\in D_2$.

Now, let
${\gamma}_{y,I(y)}$
be the shortest curve connected the points $y$ and $I(y)$. Since $y\in D_1$ and $I(y)\in D_2$, then the curves $\eta$ and
${\gamma}_{y,I(y)}$
have a non-empty intersection, i.~e. there is a point
$t\in M$ such that
$t\in {\gamma}_{y,I(y)} \cap \eta$.
Without loss of generality we may assume that
$t\in {\gamma}_{x,I(x)}$
(otherwise, one can consider the image of ${\gamma}_{x,I(x)}$ under the map $I$).

It is clear that $\rho(I(x),I(y))=\rho(x,y)=\operatorname{diam}(M)$.
Let us consider the triples of points $(x,t,y)$ and $(I(x),t,I(y))$.
According to the triangle inequality, we get
$$
\rho(x,y)\leq \rho(x,t)+\rho(t,y),
$$
$$
\rho(I(x),I(y))\leq
\rho(I(x),t)+\rho(t,I(y)).
$$
Adding these inequalities we get the following:
$$
2\operatorname{diam}(M)=\rho(x,y)+\rho(I(x),I(y))\leq
\rho(x,t)+\rho(I(x),t)+\rho(t,y)+\rho(t,I(y)).
$$
Further, using the fact that the point $t$ is situated on two shortest curves ${\gamma}_{x,I(x)}$ and ${\gamma}_{y,I(y)}$, we get
$$
\rho(x,t)+\rho(I(x),t)=\rho(x,I(x))=l({\gamma}_{x,I(x)})\leq \operatorname{diam}(M),
$$
$$
\rho(y,t)+\rho(I(y),t)=\rho(y,I(y))=l({\gamma}_{y,I(y)})\leq \operatorname{diam}(M).
$$
Consequently, three last inequalities are fulfilled if and only if they become equalities.
In particular,
$\rho(x,I(x))=\rho(y,I(y))=\operatorname{diam}(M)$,
q.e.d.
\end{proof}
\bigskip

It follows from the proved theorem that the geodesic diameter
of the surface bounding the convex centrally symmetric body
in three-dimensional space, equal to the maximum distance between pairs
of symmetric points, which greatly simplifies the calculation of this value. \footnote{\, In particular, this result was used
for the computation of the geodesic diameter of the boundary of a rectangular three-dimensional parallelepiped in the paper
{\it  Nikonorov Yu.G.,  Nikonorova Yu.V. The intrinsic diameter of the surface of a parallelepiped // Discrete and Computational Geometry, 40(4) (2008), 504--527;}
see also {\it Hess R., Grinstead Ch., Grinstead M., Bergstrand D. Hermit points on a box // College Math. J. 39(1) (2008),  12--23.}}
\medskip

It is also natural to raise the following question: {\it Is there an analogue of Theorem~1 for
length metric spaces homeomorphic to the sphere $S^q$
for some $q\geq 3$?~}\footnote{\, It follows from recent results of A.V.~Podobryaev and Yu.L.~Sachkov
({\it Podobryaev A.V., Sachkov Yu.L. Cut locus of a left invariant Riemannian metric on $SO_3$ in the axisymmetric case //
Journal of Geometry and Physics, 110 (2016), 436--453;
Podobryaev A.V. Diameter of the Berger Sphere // Mathematical Notes, 193(5) (2018), 846–-851};
see the formula (4) in the first paper and the proof of Theorem 1 in the second paper) that a three-dimensional analogue of Theorem 1
is not true even for some of the Berger spheres (the three-dimensional sphere with standard Riemannian metric compressed
along the fibers of the Hopf fibration).
The group $SU(2)$ is diffeomorphic to $S^3$ and could be considered as the group of unit quaternions
($a_0+a_1{\bf i}+a_2{\bf j}+a_3 {\bf k}$, $a_i \in \mathbb{R}$, $a_0^2+a_1^2+a_2^2+a_3^2=1$).
The center of $SU(2)$ consists of two elements $\pm 1$. The set of left-invariant metrics on $SU(2)$ is 3-parametrical, the corrresponding parameters
are the eigenvalues $I_1, I_2, I_3 >0$ of a given metric with respect of standard (Killing) metric on $SU(2)$.
The Berger spheres correspond to the case $I_1=I_2$. Due to the homogeneity, the diameter of any Berger sphere could be computed as the intrinsic distance
between $1$ and the farthest point from $1$ in~$S^3$. It should be noted that the farthest point from $1$ is exactly the point $-1$ for $I_1=I_2<2I_3$,
whereas the set of the farthest points from $1$ is the circle
$\{a_0 + a_1{\bf i}+a_2{\bf j}\, |\, a_0 = I_3/(I_3-I_1)\}$ for $I_1=I_2>2I_3$ (we thank Alexey Podobryaev for an interesting discussion).
It is clear that for any left-invariant metric on $SU(2)$ the map $a \mapsto -a$
(which is the multiplication  by the central element $-1$) is an involutive isometry (and a Clifford--Wolf translation, moving all points the same distance).
Therefore, a three-dimensional analogue of Theorem 1 is not true for this involutive isometry on every Berger metric with $I_1=I_2>2I_3$.}

Note that in the case of $q=1$ the theorem remains valid,
however, its assertion becomes trivial.

\end{document}